\documentclass[a4paper]{article}

\usepackage[a4paper,margin=1in]{geometry}
\usepackage{amsmath}
\usepackage{amsfonts}
\usepackage{hyperref}
\usepackage[all]{xy}
\usepackage{amsthm}

\newtheorem{theorem}{Theorem}

\newtheorem{definition}[theorem]{Definition}
\newtheorem{remark}[theorem]{Remark}

\newtheorem{proposition}[theorem]{Proposition}

\title{\textbf{Euler Equations of the Generalized Bott-Virasoro Group}}
\author{DING BOSHU}

\begin{document}

\maketitle

\begin{abstract}
   \noindent  In this paper, we will generalize the Bott-Virasoro group, applying the concept of the connection cochain, and derive the Euler equations corresponding to the generalized Bott-Virasoro group. We will show the relationships between the new Euler equations and the old ones. Moreover, we will study the geodesic equation corresponding to the Burgers equation and apply it to the exponential curve.
\end{abstract}

\section{Introduction}

 In Section 2, we give some fundamental knowledge of the Euler equations, 
 \begin{align*}
             \frac d{dt}m(t)=\text{\emph{ad}}^*_{A^{-1}(m(t))}m(t)
    \end{align*}
 which has also been generalized to the homogeneous space. For some details, please refer to \cite{1}. In Section 3, we apply the concept of the connection cochain, and construct the generalized Bott-Virasoro group $\text{Diff}(S^1)\times_{\chi^\alpha_\Bbb R+\beta B}+\Bbb R$, and shows that 
 \begin{proposition}
  The generalized Virasoro group $\text{Diff}(S^1)\times_{\chi^\alpha_\Bbb R+\beta B}\Bbb R$ is isomorphic to $((\widetilde{\text{Diff}(S^1)}\times_{\Bbb Z,\alpha}\Bbb R)\times_{\beta B_\Bbb R}\Bbb R)/\text{Im }\iota$, where $B_\Bbb R:(\widetilde{\text{Diff}(S^1)}\times_{\Bbb Z,\alpha}\Bbb R)\times (\widetilde{\text{Diff}(S^1)}\times_{\Bbb Z,\alpha}\Bbb R)\to\Bbb R$ is given by
  \begin{align*}
    ( [\tilde f_1,a_1],[\tilde f_2,a_2])\mapsto B(f_1,f_2),
  \end{align*}
  and $\iota:\Bbb R\to (\widetilde{\text{Diff}(S^1)}\times_{\Bbb Z,\alpha}\Bbb R)\times_{\beta B_\Bbb R}\Bbb R$ is given by
  \begin{align*}
           a\mapsto ([\text{\emph{id}},a],-a).
  \end{align*}
\end{proposition}

\noindent For the details of the connection cochain, please refer to \cite{3}. In Section 4, we talk about the Burgers Eqaution and shows for the exponential curve $\exp(tX)$ that
\begin{proposition}
     The curve $\exp(tX)$ is a geodesic if and only if $X=\text{const}\ne0$. 
\end{proposition}

\noindent In section 5, we talk about the KdV equation and apply the Euler equation to the generalized Bott-Virasoro group in Proposition \ref{2517}. Moreover, we shows that the mKdV equation is Hamiltonian with respect to some Poisson structure in Proposition \ref{2518}. In Section 6 and 7, we apply the Euler equation to the generalized Bott-Virasoro group and obtain 

\begin{proposition}
  The Euler equation corresponding to $(\text{Diff}(S^1)\times_{\chi_\Bbb R^\alpha+\beta B}\Bbb R)/S^1$ is
  \begin{align*}
     \frac d{dt}(u,a)&=\text{ad}^*_{((-\partial_x^2)^{-1}u,a)}(udx\otimes dx,a) \\
     &=(-2((-\partial_x^2)^{-1}u)'u-((-\partial_x^2)^{-1}u)u'-\alpha a((-\partial_x^2)^{-1}u)'''+\beta a((-\partial_x^2)^{-1}u)',a).
  \end{align*}
  Letting $u=v''$, we have
  \begin{align*}
         v_{xxt}=2v_xv_{xx}+vv_{xxx}+\alpha av_{xxx}-\beta av_x,
  \end{align*}
  which is a new equation. Letting $v=w-\alpha a$, we have
  \begin{align*}
     v_{xxt}=2v_xv_{xx}+vv_{xxx}-\beta av_x.
  \end{align*}
\end{proposition}

\begin{proposition}
The Euler equation corresponding to the generalized Virasoro group and $H^1$-metric is
\begin{align*}
   \frac d{dt} (u,a)&=\text{ad}^*_{((1-\partial_x^2)^{-1}u\frac d{dx},a)}(udx\otimes dx,a) \\
   &=-2(((1-\partial_x^2)^{-1}u)'u-((1-\partial_x^2)^{-1}u)u'-\alpha a((1-\partial_x^2)^{-1}u)'''+\beta a((1-\partial_x^2)^{-1}u)',a)
\end{align*}
Letting $u=v-v''$, we have
\begin{align*}
   v_t-v_{xxt}=-3v_xv+2v_xv_{xx}+vv_{xxx}-\alpha av_{xxx}+\beta av_x.
\end{align*}
Letting $v=w+\frac13a\beta$, we obtain
\begin{align*}
    w_t-w_{xxt}=-3w_xw+2w_xw_{xx}+ww_{xxx}-(\alpha a+\frac13\beta a)w_{xxx}.
\end{align*}
\end{proposition}

\section{The Euler Equation}

Let $G$ be a Lie group and $\mathfrak g$ its Lie algebra.
\begin{definition}
\normalfont
    A linear invertible operator $A:\mathfrak g\to\mathfrak g^*$ is called an \emph{\textbf{inertia operator}} if
    \begin{align*}
        A(u)(v)=A(v)(u)
    \end{align*}
    for any $u,v\in\mathfrak g$.
\end{definition}

\noindent The inertia operator gives rise to an inner product $\langle-,-\rangle:\mathfrak g\times\mathfrak g\to \Bbb R$ given by
\begin{align*}
       (u,v)\mapsto A(u)(v)
\end{align*}
for $u,v\in\mathfrak g$ which leads to a right-invariant metric $(-,-):TG\otimes TG\to\Bbb R$ defined by 
\begin{align*}
(X,Y):=\langle\theta(X),\theta(Y)\rangle,
\end{align*}
where $X,Y\in T_gG$ and $\theta:TG\to\mathfrak g$ is the right Mauer-Cartan form given by
\begin{align*}
      X\mapsto {r_{g^{-1}}}_*X
\end{align*}
for $X\in T_gG$. Recall that the energy functional of a smooth curve $c:[a,b]\to G$ is given by
\begin{align*}
    E(c)=\frac12\int_a^b(c'(t),c'(t))dt.
\end{align*}

\begin{theorem}
    A curve $c:[a,b]\to G$ is a geodesic with respect to the right-invariant metric on $G$ if and only if 
    \begin{align*}
             \frac d{dt}m(t)=\text{\emph{ad}}^*_{A^{-1}(m(t))}m(t)
    \end{align*}
where $m(t)=A(\theta(c'(t)))$ and $\text{\emph{ad}}^*$ represents the coadjoint action of $\mathfrak g$ on $\mathfrak g$.
\end{theorem}

\noindent Let $G$ be a Lie group and $K$ its subgroup. Consider the space $G/K$ of right cosets $\{Kg\mid g\in G\}$ on which the group $G$ acts on the right.

\begin{definition}
\normalfont
   An linear operator $A:\mathfrak g\to\mathfrak g^*$ is called a \textbf{\emph{degenerate inertia operator}} if
   \begin{align*}
      A(u)(v)=A(v)(u)
   \end{align*}
   for any $u,v\in\mathfrak g$.
\end{definition}

\noindent Let $\langle -,-\rangle:\mathfrak g\times\mathfrak g\to\Bbb R$ be the degenerate inner product given by
\begin{align*}
       (u,v)\mapsto A(u)(v)
\end{align*}
for $u,v\in\mathfrak g$ and $(-,-)_G:TG\otimes TG\to\Bbb R$ the corresponding degenerate right-invariant metric given by
\begin{align*}
    (X,Y)\mapsto \langle \theta(X),\theta(Y)\rangle
\end{align*}
for $X,Y\in T_gG$.

\begin{theorem}
    The degenerate right-invariant metric $(-,-)_G$ on $G$ descends to a metric on $G/K$ if $\ker A=\mathfrak k$ and
    \begin{align}\label{09253}
        A(\text{\emph{Ad}}_ku)(\text{\emph{Ad}}_kv)=A(u)(v)
    \end{align}
    for all $k\in\mathfrak k$ and $u,v\in\mathfrak g$.
\end{theorem}

\noindent Note that in metric $(-,-):T(G/K)\otimes T(G/K)\to\Bbb R$ induced by $(-,-)_G$ is given by
\begin{align} \label{09251}
        (X,Y)\mapsto (\tilde X,\tilde Y)_G
\end{align}
for $X,Y\in T_{Kg}(G/K)$, where $\tilde X,\tilde Y\in T_gG$ are some lifts of $X,Y$.

\begin{theorem}
    A curve $c:[a,b]\to G/K$ is a geodesic with respect to the metric given by (\ref{09251}) if and only if there exists a lift $\tilde c:[a,b]\to G$ such that
    \begin{align}\label{09252}
        \frac d{dt}m(t)=\text{\emph{ad}}^*_{A^{-1}(m(t))}m(t)
    \end{align}
    where $m(t)=A(\theta(\tilde c'(t)))$.
\end{theorem}

\noindent   Note that the right-hand side of (\ref{09252}) is well-defined. Indeed, by (\ref{09253}), for any $V\in\mathfrak k$ and $a,b\in\mathfrak g$, we have
\begin{align*}
        A(\text{ad}_Va)(b)+A(a)(\text{ad}_Vb)=\frac {d}{dt}\bigg|_{t=0} A(\text{Ad}_{c_V(t)}a)(\text{Ad}_{c_V(t)}b)=0,
\end{align*}
where $c_V$ is a curve in $K$ with initial vector $V$. It follows that
\begin{align*}
   (\text{ad}^*_VA(a))(b)=-A(a)(\text{ad}_Vb)=A(a)(\text{ad}_bV)=-A(\text{ad}_ba)(V)=-A(V)(\text{ad}_ba)=0,
\end{align*}
for any $V\in\mathfrak k$ and $a,b\in\mathfrak g$.

\section{Connection Cochain and Generalized Virasoro Group}

\begin{definition}
\normalfont
    For a central extension $\xymatrix{1 \ar[r] &  A \ar[r] &  \tilde G \ar[r]^\pi &  G \ar[r] &  1}$, a cochain $\tau:\tilde G\to A$ that satisfies the condition 
    \begin{align*}
        \tau(a\tilde g)=a+\tau(\tilde g)
    \end{align*}
    for all $\tilde g\in\tilde G$ and $a\in A$ is called a \textbf{\emph{connection cochain}}.
\end{definition}

\begin{proposition}
   There exists a $2$-cocycle $\sigma$ on $G$ such that $\delta\tau(\tilde g,\tilde h)=\sigma(g,h)$ for any $g,h\in G$. We call $\sigma$ the \textbf{\emph{curvature}} of a connection cochain.
\end{proposition}

\begin{proposition}
          Define the \textbf{\emph{Euler cocycle}} $\chi:G\times G\to A$ by
          \begin{align*}
             \chi(g_1,g_2)=s(g_1)s(g_2)s(g_1g_2)^{-1}
          \end{align*}
          where $s:G\to\tilde G$ is a section of $\pi:\tilde G\to G$. Then, we have $[\chi]=-[\sigma]$ in $H^2(G,A)$.
\end{proposition}

\begin{proposition} \label{09256}
   Let $s:G\to\tilde G$ be a section, which gives rise to an Euler cocycle $\chi$, and $\tau:\tilde G\to A$ a connection cochain defined by
   \begin{align*}
      \tilde g\mapsto \tilde gs(g)^{-1}.
   \end{align*}
   Then, we have $\chi=-\sigma$.
\end{proposition}

\begin{remark} \label{09257}
\normalfont
    Let $B$ be an abelian group with a homomorphism $\iota:A\to B$. A map $\tau:\tilde G\to B$ satisfying
    \begin{align*}
        \tau(a\tilde g)=\tau(\tilde g)+\iota(a)
    \end{align*}
    is called a connection cochain with values in $B$.  Given a central extension
    \begin{align*}
       \xymatrix{1 \ar[r] &  A \ar[r] &  \tilde G \ar[r] &  G \ar[r] &  1},
    \end{align*}
    we can extend it to a central extension
    \begin{align}\label{09254}
        \xymatrix{1 \ar[r] &  B \ar[r] &  \tilde G\times _{A,\iota}B \ar[r] &  G \ar[r] &  1}.
    \end{align}
    The multiplication in $\tilde G\times_{A,\iota}B$ is given by
    \begin{align*}
       [\tilde g_1,b_1][\tilde g_2,b_2]=[\tilde g_1\tilde g_2,b_1+b_2].
    \end{align*}
    Define $\tau_B:\tilde G\times_{A,\iota}B\to B$ by
    \begin{align*}
       [\tilde g,b]\mapsto \tau(\tilde g)+b.
    \end{align*}
    It is well-defined since
    \begin{align*}
       \tau(a\tilde g)+b-\iota(a)=\tau(\tilde g)+b.
    \end{align*}
    Moreover, since
    \begin{align*}
       \tau_B([\tilde g,b]b')=\tau([\tilde g,b+b'])=\tau(\tilde g)+b+b'=\tau_B([\tilde g,b])+b',
    \end{align*}
    we see that $\tau_B$ is a connection cochain of (\ref{09254}).
\end{remark}

\noindent Consider the following central extension:
\begin{align*}
     \xymatrix{
       1 \ar[r] &  \Bbb Z \ar[r]  &   \widetilde{\text{Diff}(S^1)}\ar[r]^\rho&  \text{Diff}(S^1)\ar[r] & 1
     }.
\end{align*}
Here, an element $\tilde f\in\widetilde{\text{Diff}(S^1)}$ is regarded as an orientation-preserving diffeomorphism of $\Bbb R$ such that $\tilde f(x+2\pi)=\tilde f(x)+2\pi$. Then, 
\begin{align*}
     \text{Diff}(S^1)=\widetilde{\text{Diff}(S^1)}/\sim,
\end{align*}
where $\tilde f\sim\tilde g$ if and only if $\tilde f(x)=\tilde g(x)+2n\pi$ for some $n\in\Bbb N$. For $\alpha\in\Bbb R$, define $\alpha:\Bbb Z\to \Bbb R$ by
\begin{align*}
     k\mapsto -2\pi^2 \alpha k.
\end{align*}
Define $\tau^\alpha:\widetilde{\text{Diff}(S^1)}\to\Bbb R$ by
\begin{align*}
   \tau^\alpha(f)=-\frac{\alpha}{2}\int_0^{2\pi}\tilde f(x)dx+\pi^2\alpha.
\end{align*}
It's easy to see that $\tau^\alpha$ is a connection cochain with values in $\Bbb R$, since
\begin{align*}
   \tau^\alpha(\tilde f\circ k)=-\frac{\alpha}{2} \int_0^{2\pi} \tilde f(x+2k\pi)dx+\pi^2\alpha=\tau^\alpha(f)+\alpha(k).
\end{align*}
By Remark \ref{09257}, $\tau^\alpha$ gives rise to a connection cochain $\tau_{\Bbb R}^\alpha:\widetilde{\text{Diff}(S^1)}\times_{\Bbb Z,\alpha}\Bbb R\to\Bbb R$ of the following central extension
\begin{align*}
     \xymatrix{
       1 \ar[r] &  \Bbb R \ar[r]  &   \widetilde{\text{Diff}(S^1)}\times_{\Bbb Z,\alpha}\Bbb R\ar[r]^-\pi&  \text{Diff}(S^1)\ar[r] & 1
     }.
\end{align*}
by $\tau_{\Bbb R}^\alpha([\tilde f,a])=\tau^\alpha(f)+a$. Define $s_\Bbb R^\alpha:\text{Diff}(S^1)\to\widetilde{\text{Diff}(S^1)}\times_{\Bbb Z,\alpha}\Bbb R$ by
\begin{align*}
   f\mapsto [\tilde f,a-\tau_{\Bbb R}^\alpha([\tilde f,a])].
\end{align*}
The section $s_\Bbb R^\alpha$ is well-defined, since
\begin{align*}
    [\tilde f\circ k,a+a'-\tau^\alpha_\Bbb R([\tilde f\circ k,a+a'])]=[\tilde f,a-\tau^\alpha_\Bbb R([\tilde f,a])].
\end{align*}
Moreover, $s^\alpha_\Bbb R$ is smooth since
\begin{align*}
      [\tilde f\circ k,a-\tau_\Bbb R^\alpha([\tilde f\circ k,a])]=[\tilde f,a-\tau_\Bbb R^\alpha([\tilde f,a])].
\end{align*}
Then, $s^\alpha_\Bbb R$ leads to the smooth Euler cocycle $\chi^\alpha_\Bbb R:\text{Diff}(S^1)\times\text{Diff}(S^1)\to\Bbb R$
\begin{align*}
    \chi_\Bbb R^\alpha(f_1,f_2)&=s^\alpha_\Bbb R(f_1,f_2)^{-1}=s^\alpha_\Bbb R(f_1f_2)^{-1}s^\alpha_\Bbb R(f_1)s^\alpha_\Bbb R(f_2) \\
    &=[\tilde f_2^{-1}\circ \tilde f_1^{-1},\tau^\alpha_\Bbb R([\tilde f \circ \tilde f_2,0])][\tilde f_1.-\tau^\alpha_\Bbb R([\tilde f_1,0])][\tilde f_2,-\tau_\Bbb R^\alpha([\tilde f_2,0])] \\
    &=\frac{\alpha}{2}\int_0^{2\pi}(\tilde f_1(x)+\tilde f_2(x)-\tilde f_1\circ\tilde f_2(x))dx-\pi^2\alpha,
\end{align*}
where we have identified $\Bbb R$ with with its image in $\widetilde{\text{Diff}(S^1)}\times_{\Bbb Z,\alpha}\Bbb R$. Note that by Proposition \ref{09256}, we have
\begin{align*}
  \chi^\alpha_\Bbb R(f_1,f_2)=-\delta\tau_\Bbb R^\alpha([\tilde f_1,a_1],[\tilde f_2,a_2]).
\end{align*}

\begin{proposition} \label{092514}
  The \textbf{\emph{generalized Virasoro group}} $\text{Diff}(S^1)\times_{\chi^\alpha_\Bbb R+\beta B}\Bbb R$ is isomorphic to $((\widetilde{\text{Diff}(S^1)}\times_{\Bbb Z,\alpha}\Bbb R)\times_{\beta B_\Bbb R}\Bbb R)/\text{Im }\iota$, where $B_\Bbb R:(\widetilde{\text{Diff}(S^1)}\times_{\Bbb Z,\alpha}\Bbb R)\times (\widetilde{\text{Diff}(S^1)}\times_{\Bbb Z,\alpha}\Bbb R)\to\Bbb R$ is given by
  \begin{align*}
    ( [\tilde f_1,a_1],[\tilde f_2,a_2])\mapsto B(f_1,f_2),
  \end{align*}
  and $\iota:\Bbb R\to (\widetilde{\text{Diff}(S^1)}\times_{\Bbb Z,\alpha}\Bbb R)\times_{\beta B_\Bbb R}\Bbb R$ is given by
  \begin{align*}
           a\mapsto ([\text{\emph{id}},a],-a).
  \end{align*}
\end{proposition}

\begin{proof}  Note that $\text{Im }\iota$ is central in $(\widetilde{\text{Diff}(S^1)}\times_{\Bbb Z,\alpha}\Bbb R)\times_{\beta B_\Bbb R}\Bbb R$, since 
\begin{align*}
    ([\text{id},a]-a)([\tilde f,b],c)=([\tilde f,a+b],c-a)=([\tilde f,b],c)([\text{id},a],-a).
\end{align*}
Therefore, $((\widetilde{\text{Diff}(S^1)}\times_{\Bbb Z,\alpha}\Bbb R)\times_{\beta B_\Bbb R}\Bbb R)/\text{Im }\iota$ is indeed a Lie group. Define $\Psi:((\widetilde{\text{Diff}(S^1)}\times_{\Bbb Z,\alpha}\Bbb R)\times_{\beta B_\Bbb R}\Bbb R)/\text{Im }\iota\to\text{Diff}(S^1)\times_{\chi^\alpha_\Bbb R+\beta B}\Bbb R$ by
\begin{align*}
    [[\tilde f,a],b]\mapsto (f,b+\tau^\alpha_\Bbb R([\tilde f,a])).
\end{align*}
Since
\begin{align*}
   \Psi([[\tilde f_1,a_1],b_1][[\tilde f_2,a_2],b_2])&=\Psi([[\tilde f_1\circ \tilde f_2,a_1+a_2],b_1+b_2+\beta B_\Bbb R([\tilde f_1,a_1],[\tilde f_2,a_2])]) \\
   &=(f_1\circ f_2,b_1+b_2+\beta B(f_1,f_2)+\tau_\Bbb R^\alpha([\tilde f_1,a_1],[\tilde f_2,a_2])) \\
   &=(f_1\circ f_2,b_1+b_2+\beta B(f_1,f_2)+\chi_\Bbb R^\alpha(f_1,f_2)+\tau_\Bbb R^\alpha([\tilde f_1,a_1])+\tau_\Bbb R^\alpha([\tilde f_2,a_2])) \\
   &=(f_1,b_1+\tau_\Bbb R^\alpha([\tilde f_1,a_1]))(f_2,b_2+\tau^\alpha_\Bbb R([\tilde f_2,a_2])) \\
   &=\Psi([[\tilde f_1,a_1],b_1])\tau_\Bbb R^\alpha([[\tilde f_2,a_2],b_2]),
\end{align*}
we see that $\Psi$ is a homomorphism. On the other hand, define $\Phi:\text{Diff}(S^1)\times_{\chi_\Bbb R^\alpha+\beta B}\Bbb R\to ((\widetilde{\text{Diff}(S^1)}\times_{\Bbb Z,\alpha}\Bbb R)\times_{\beta B_\Bbb R}\Bbb R)/\text{Im }\iota$ by
\begin{align*}
   (f,b)\mapsto[s^\alpha_\Bbb R(f),b].
\end{align*}
It's easy to see that
\begin{align*}
   \Phi\circ\Psi([\tilde f,a],b)=\Phi(f,b+\tau_\Bbb R^\alpha([\tilde f,a]))=[[\tilde f,a-\tau^\alpha_\Bbb R([\tilde f,a])],b+\tau_\Bbb R^\alpha([\tilde f,a])]=[[\tilde f,a],b],
\end{align*}
and
\begin{align*}
   \Psi\circ\Phi(f,b)=\Psi([s_\Bbb R^\alpha(f),b])=(f,b+\tau_\Bbb R^\alpha(s_\Bbb R^\alpha(f)))=(f,b),
\end{align*}
which implies the claim.
\end{proof}

\noindent Recall that the Lie algebraic Euler cocycle $e:\mathfrak X(S^1)\times \mathfrak X(S^1)\to \Bbb R$ is defined by
\begin{align*}
      e(u\frac d{dx},v\frac d{dx})=\int_{S^1}uv'dx.
\end{align*}
Since
\begin{align*}
    \frac {d^2}{dtds}\bigg|_{t=s=0}\chi_\Bbb R^\alpha(c_{u\frac d{dx}}(-,t),c_{v\frac d{dx}}(-,s))&=-\frac {d^2}{dtds}\bigg|_{t=s=0} \frac \alpha 2\int_0^{2\pi} c_{u\frac d{dx}}(c_{v\frac d{dx}}(x,s),t)dx =\frac\alpha2\int_{S^1}uv'dx,
\end{align*}
we see that
\begin{align*}
   \text{Lie}(\text{Diff}(S^1)\times_{\chi_\Bbb R^\alpha+\beta B}\Bbb R)=\chi(S^1)\times_{\alpha e+\beta\omega}\Bbb R
\end{align*}
where $\omega$ denotes the Gelfand-Fuchs cocycle.

\section{Burgers Equation}

\noindent Let $G=\text{Diff}(S^1)$ be the orientation-preserving diffeomorphism group of $S^1$, whose Lie algebra is $\mathfrak X(S^1)$. The dual space of $\mathfrak X(S^1)$ is regarded as $\{udx\otimes dx\mid u\in C^\infty(S^1)\}$ with the following pair
\begin{align*}
      ( udx\otimes dx)(v\frac d{dx})=\int_{S^1}uvdx.
\end{align*}
Define $A:\mathfrak X(S^1)\to\mathfrak X(S^1)^*$ be the inertia operator by
\begin{align*}
       u\frac{d}{dx}\mapsto udx\otimes dx.
\end{align*}
Let $\langle-,-\rangle:\mathfrak X(S^1)\times\mathfrak X(S^1)$ be the inner product associated with $A$ and $(-,-)$ the corresponding right-invariant metric. Since
\begin{align*}
     (\text{ad}^*_{u\frac d{dx}}(vdx\otimes dx))(w\frac d{dx})&=-(vdx\otimes dx)(\text{ad}_{u\frac d{dx}}w\frac d{dx})=-(vdx\otimes dx)((u'w-uw')\frac d{dx}) \\
     &=\int_{S^1} (-2u'v-uv')wdx \\
     &=((-2u'v-uv')dx\otimes dx)(w\frac d{dx}),
\end{align*}
we obtain the Euler equation of $\text{Diff}(S^1)$ with respect to the right-invariant metric $(-,-)$:
\begin{align} \label{09257}
         \frac d{dt} (udx\otimes dx)=\text{ad}^*_{A^{-1}(udx\otimes dx)}(udx\otimes dx)=\text{ad}^*_{u\frac d{dx}}(udx\otimes dx)=-3u'u dx\otimes dx,
\end{align}
where $udx\otimes dx$ is a curve on $\mathfrak X(S^1)^*$. For simplicity, we write  (\ref{09257})  as
\begin{align} \label{092510}
      u_t=-3u_xu,
\end{align} 
called the \textbf{\emph{Burgers equation}}. For $c$ be a curve on $\text{Diff}(S^1)$, since
\begin{align} \label{09259}
     u(-,t)={r_{c(-,t)^{-1}}}_* c_t(-,t)=\frac d{ds}\bigg|_{s=t} c(-,s)\circ c(-,t)^{-1}=c_t(-,t)\circ c(-,t)^{-1},
\end{align}
substituting (\ref{09259}) into (\ref{092510}), we obtain the geodesic equation
\begin{align} \label{092511}
     2c_{tx}c_t+c_{tt}c_x=0.
\end{align}

\begin{proposition}
     The curve $\exp(tX)$ is a geodesic if and only if $X=\text{const}\ne0$. 
\end{proposition}

\begin{proof}
   Since 
   \begin{align*}
      \frac d{dt} \varphi_{\tilde X}(1_G,st)=s\tilde X_{\varphi_{\tilde X}(1_G,st)}=\widetilde{sX}_{\varphi_{\tilde X}(1_G,st)},
   \end{align*}
   we have
   \begin{align*}
           \varphi_{\tilde X}(1_G,st)=\varphi_{\widetilde{sX}}(1_G,t).
   \end{align*}
   It follows that
   \begin{align*}
           \exp(tX)=\varphi_{\widetilde {tX}}(1_G,1)=\varphi_{\tilde X}(1_G,t).  
   \end{align*}
   It's easy to see that
   \begin{align*}
      \frac d{dt}\varphi_{\tilde X}(1_G,t)=\tilde X_{\varphi_{\tilde X}(1_G,t)}={l_{\varphi_{\tilde X}(1_G,t)}}_*X=\frac d{ds}\bigg|_{s=0}\varphi_{\tilde X}(1_G,t)\circ c_X(s)=\varphi_{\tilde X}(1_G,t)'X.
   \end{align*}
   Therefore, we have
   \begin{align*}
       \frac {d^2}{dt^2} \varphi_{\tilde X}(1_G,t)&=\frac d{dt} \varphi_{\tilde X}(1_G,t)'X=(\varphi_{\tilde X}(1_G,t)'X)'X=\varphi_{\tilde X}(1_G,t)''X^2+\varphi_{\tilde X}(1_G,t)'X'X,
   \end{align*}
   and
   \begin{align*}
    \frac{d^2}{dxdt}\varphi_{\tilde X}(1_G,t)=\frac d{dx}\varphi_{\tilde X}(1_G,t)'X=\varphi_{\tilde X}(1_G,t)''X+\varphi_{\tilde X}(1_G,t)'X'.
   \end{align*}
   Substituting these results into (\ref{092511}), we have
   \begin{align*}
      2(\varphi_{\tilde X}(1_G,t)''X+\varphi_{\tilde X}(1_G,t)'X')(\varphi_{\tilde X}(1_G,t)'X)+(\varphi_{\tilde X}(1_G,t)''X^2+\varphi_{\tilde X}(1_G,t)'X'X)\varphi_{\tilde X}(1_G,t)'=0,
   \end{align*}
   which is equivalent to
   \begin{align*}
       (\varphi_{\tilde X}(1_G,t)'X)'=0.
   \end{align*}
   It follows that $\varphi_{\tilde X}(1_G,t)'X=\text{const}$. Therefore, we have
   \begin{align*}
      X=\tilde X_{\varphi_{\tilde X}(1_G,0)}=\varphi_{\tilde X}(1_G,0)'X=\text{const}.
   \end{align*}   
\end{proof}

\section{Generalized KdV Equation and mKdV Eqaution}

Let $G=\text{Diff}(S^1)\times_B\Bbb R$, where $B:\text{Diff}(S^1)\times\text{Diff}(S^1)\to\Bbb R$ is the Bott cocycle:
\begin{align*}
   (\varphi,\psi)\mapsto\frac12\int_{S^1} \log (\varphi\circ\psi)' d\log\psi',
\end{align*}
the Lie algebra of which is $\mathfrak X(S^1)\times_\omega\Bbb R$, where $\omega:\mathfrak X(S^1)\times\mathfrak X(S^1)\to\Bbb R$ is the Gelfand-Fuchs cocycle:
\begin{align*}
     (u\frac d{dx},v\frac d{dx})\mapsto \int_{S^1}v'w''dx.
\end{align*}
The dual space of $\mathfrak X(S^1)\times_\omega\Bbb R$ can be regarded as $\{(udx\otimes dx,a)\mid u\in C^\infty(S^1),a\in \Bbb R\}$ with the following pair:
\begin{align*}
    (udx\otimes dx,a)(v\frac d{dx},b)=\int_{S^1}uvdx+ab.
\end{align*}
Define the inertia operator $A:\mathfrak X(S^1)\times_\omega\Bbb R\to(\mathfrak X(S^1)\times_\omega\Bbb R)^*$ by
\begin{align*}
          (u\frac d{dx},a)\mapsto (udx\otimes dx,a).
\end{align*}
Let $\langle-,-\rangle:(\mathfrak X(S^1)\times_\omega\Bbb R)\times(\mathfrak X(S^1)\times_\omega\Bbb R)\to\Bbb R$ be the inner product associated with $A$ and $(-,-)$ the corresponding right-invariant metric on $\text{Diff}(S^1)\times_B\Bbb R$. Since
\begin{align*}
     (\text{ad}^*_{(u\frac d{dx},a)}(vdx\otimes dx,b))(w\frac d{dx},c)&=-(vdx\otimes dx,b)(\text{ad}_{(u\frac d{dx},a)}(w\frac d{dx},c))\\
     &=-(vdx\otimes dx,b)((u'w-uw')\frac d{dx},\int_{S^1}u'w''dx) \\
     &=\int_{S^1} (-2u'v-uv'-bu''')wdx \\
     &=((-2u'v-uv'-bu''')dx\otimes dx,0)(w\frac d{dx},c),
\end{align*}
we obtain the Euler equation of $\text{Diff}(S^1)\times_B\Bbb R$:
\begin{align}\label{092513}
\begin{split}
     \frac{d}{dt} (udx\otimes dx,a)&=\text{ad}^*_{A^{-1}(udx\otimes dx,a)}(udx\otimes dx,a)=\text{ad}^*_{(u\frac d{dx},a)}(udx\otimes dx,a)\\
     &=((-3u'u-au''')dx\otimes dx,0).
\end{split}
\end{align}
where $(udx\otimes dx,a)$ is a curve on $(\mathfrak X(S^1)\times_\omega\Bbb R)^*$. For similicity, we write (\ref{092513})  as
\begin{align*}
       u_t=-3u_xu-au_{xxx},
\end{align*}
which is the \textbf{\emph{KdV equation}}. The KdV equation can also be obtained by the Hamiltonian equation with respect to the constant Poisson structure $\{-,-\}:C^\infty((\mathfrak X(S^1)\times_\omega\Bbb R)^*)\times C^\infty((\mathfrak X(S^1)\times_\omega\Bbb R)^*)\to C^\infty((\mathfrak X(S^1)\times_\omega\Bbb R)^*)$:
\begin{align*}
    \{F,G\}(udx\otimes dx,a)=(-\frac12dx\otimes dx,0)[dF(udx\otimes dx,a),dG(udx\otimes dx,a)]
\end{align*}
where $dF(udx\otimes dx,a)$ and $dG(udx\otimes dx,a)$ are regarded as elements in $\mathfrak X(S^1)\times_\omega\Bbb R$. Recall that for any $H$ and $f$ in $C^\infty((\mathfrak X(S^1)\times_\omega\Bbb R)^*)$, we have
\begin{align*}
      (X_Hf)(udx\otimes dx,a)&=\{H,f\}(udx\otimes dx,a)=(-\frac12dx\otimes dx,0)[dH(udx\otimes dx,a),df(udx\otimes dx,a)] \\
      &=-(\text{ad}^*_{dH(udx\otimes dx,a)}(-\frac12dx\otimes dx,0))(df(udx\otimes dx,a)).
\end{align*}
It follows that in this case the \textbf{Hamiltonian function} is
\begin{align*}
    \frac d{dt} (udx\otimes dx,a)=X_H(udx\otimes dx,a)=-\text{ad}^*_{dH(udx\otimes dx,a)}(-\frac12dx\otimes dx,0).
\end{align*}
Define $H:(\mathfrak X(S^1)\times_\omega\Bbb R)^*\to\Bbb R$ by
\begin{align*}
    (udx\otimes dx,a)\mapsto \int_{S^1} (\frac12u^3-\frac a2u_x^2)dx.
\end{align*}
It's easy to see that
\begin{align*}
    (vdx\otimes dx,b) (dH(udx\otimes dx,a))&=\frac d{dt}\bigg|_{t=0} H((udx\otimes dx,a)+t(vdx\otimes dx,b)) \\
    &=\frac d{dt}\bigg|_{t=0} \int_{S^1}( \frac12(u+tv)^3-\frac{a+tb}2(u+tv)_x^2 )dx \\
    &=\int_{S^1} (\frac32u^2v-\frac b2u_x^2-au_xv_x)dx \\
    &=(vdx\otimes dx,b)((\frac32u^2+au_{xx})\frac d{dx},-\frac12\int_{S^1}u_x^2dx).
\end{align*}
It follows that the Hamiltonian fucntion with respect to $H$ is
\begin{align*}
        \frac d{dt}(udx\otimes dx,a)&=-\text{ad}^*_{((\frac32u^2+au_{xx})\frac d{dx},-\frac12\int_{S^1}u_x^2dx)}(-\frac12dx\otimes dx,0)\\
        &=(-(\frac32u^2+au_{xx})_xdx\otimes dx,0) \\
        &=((-3u_xu-au_{xxx})dx\otimes dx,0),
\end{align*}
which is the KdV equation. With this method, we can also obtain the \textbf{mKdV equation}. Define $H:(\mathfrak X(S^1)\times_\omega\Bbb R)^*\to\Bbb R$ by
\begin{align*}
          (udx\otimes dx,a)\mapsto\int_{S^1}(\frac34u^4-\frac a2u_x^2 )dx.
\end{align*}
It's easy to see that
\begin{align*}
  (vdx\otimes dx,b)(dH(udx\otimes dx,a))&=\frac d{dt}\bigg|_{t=0}H((udx\otimes dx,a)+t(vdx\otimes dx,b)) \\
  &=\frac d{dt}\bigg|_{t=0}\int_{S^1} (\frac34(u+tv)^4-\frac{a+tb}{2}(u+tv)_x^2)dx\\
  &=\int_{S^1} (3u^3v-\frac b2u_x^2-au_xv_x)dx \\
  &=(vdx\otimes dx,b)((3u^3+au_{xx})\frac d{dx},-\frac12\int_{S^1}u_x^2dx).
\end{align*}
It follows that the Hamiltonian function with respect to $H$ is
\begin{align*}
    \frac d{dt}(udx\otimes dx,a)&=-\text{ad}^*_{((3u^3+au_{xx})\frac d{dx},-\frac12\int_{S^1}u_x^2dx)}(-\frac 12dx\otimes dx,0) \\
    &=(-(3u^3+au_{xx})_xdx\otimes dx,0) \\
    &=((-6u^2u_x-au_{xxx})dx\otimes dx,0),
\end{align*}
which is the mKdV equation. 

\begin{proposition} \label{2518}
   The mKdV equation is Hamiltonian on the dual $(\mathfrak X(S^1)\times_\omega\Bbb R)^*$ associated to the constant Poisson structure with the freezing point in $(\mathfrak X(S^1)\times_\omega\Bbb R)^*$ is $(-\frac12dx\otimes dx,0)$.
\end{proposition}

\begin{proposition} \label{2517}
   The Euler equation corresponding to the generalized Virasoro group is
   \begin{align*}
          \frac d{dt}(udx\otimes dx,a)&=\text{ad}^*_{A^{-1}(udx\otimes dx,a)}(udx\otimes dx,a)=\text{ad}^*_{(u\frac d{dx},a)}(udx\otimes dx,a) \\
          &=((-3uu'-a\alpha u'''+a\beta u')dx\otimes dx,a).
   \end{align*}
\end{proposition}

\noindent  Let $u=v+\frac13a\beta$, we obtain
\begin{align*}
    v_t=-3v_xv-a\alpha v_{xxx},
\end{align*}
which is the KdV euqation.

\section{Generalized Camassa–Holm equation}

\noindent Let $G=\text{Diff}(S^1)\times_B\Bbb R$, whose Lie algebra is $\mathfrak X(S^1)\times_\omega\Bbb R$. Define the inertia operator $A:\mathfrak X(S^1)\times_\omega\Bbb R\to(\mathfrak X(S^1)\times_\omega\Bbb R)^*$ by
\begin{align*}
     (u\frac d{dx},a)\mapsto ((u-u_{xx})dx\otimes dx,a).
\end{align*}
Let $\langle-,-\rangle$ be the inner product associated with $A$ and $(-,-)$ the corresponding right-invariant metric. Recall that
\begin{align*}
    \text{ad}^*_{(u\frac d{dx},a)}(vdx\otimes dx,b)=((-2u'v-uv'-bu''')dx\otimes dx,0).
\end{align*}
We see that the associated Euler equation is
\begin{align*}
      \frac d{dt}(udx\otimes dx,a)&=\text{ad}^*_{A^{-1}(udx\otimes dx)}(udx\otimes dx,a)=\text{ad}^*_{((1-\frac {d^2}{dx^2})^{-1}u\frac d{dx},a)}(udx\otimes dx,a) \\
      &=((-2((1-\frac {d^2}{dx^2})^{-1}u)'u-((1-\frac {d^2}{dx^2})^{-1}u)u'-b((1-\frac {d^2}{dx^2})^{-1}u)''')dx\otimes dx,0)
\end{align*}
Let $u=v-v''$. Then, the equation above becomes
\begin{align*}
    \frac d{dt}(v-v_{xx})=-2v_x(v-v_{xx})-v(v-v_{xx})_x-bv_{xxx},
\end{align*}
which is equivalent to
\begin{align*}
  v_t-v_{txx}=-3v_xv+2v_xv_{xx}+vv_{xxx}-bv_{xxx},
\end{align*}
called the \textbf{\emph{Camassa–Holm equation}}. Note that in the discussion above, we have restricted our curve to the form
\begin{align*}
   u=v-v''.
\end{align*}

\begin{proposition}
  The Euler equation corresponding to $(\text{Diff}(S^1)\times_{\chi_\Bbb R^\alpha+\beta B}\Bbb R)/S^1$ is
  \begin{align*}
     \frac d{dt}(u,a)&=\text{ad}^*_{((-\partial_x^2)^{-1}u,a)}(udx\otimes dx,a) \\
     &=(-2((-\partial_x^2)^{-1}u)'u-((-\partial_x^2)^{-1}u)u'-\alpha a((-\partial_x^2)^{-1}u)'''+\beta a((-\partial_x^2)^{-1}u)',a).
  \end{align*}
  Letting $u=v''$, we have
  \begin{align*}
         v_{xxt}=2v_xv_{xx}+vv_{xxx}+\alpha av_{xxx}-\beta av_x,
  \end{align*}
  which is a new equation. Letting $v=w-\alpha a$, we have
  \begin{align*}
     v_{xxt}=2v_xv_{xx}+vv_{xxx}-\beta av_x.
  \end{align*}
\end{proposition}

\section{Generalized Hunter-Saxton Equation}

\noindent Let $G=(\text{Diff}(S^1)\times_B\Bbb R)/S^1$ and $A:\mathfrak X(S^1)\times_\omega\Bbb R\to(\mathfrak X(S^1)\times_\omega\Bbb R)^*$ given by
\begin{align*}
   (u\frac d{dx},a)\mapsto(-u''dx\otimes dx,a).
\end{align*}
It's easy to see that $\ker A=\Bbb R=\text{Lie}(S^1)$, and
\begin{align*}
     A(\text{Ad}_{(0,k)}(u\frac d{dx},a))(\text{Ad}_{(0,k)}(v\frac d{dx},b))=A(u\frac d{dx},a)(v\frac d{dx},b).
\end{align*}
It follows that the degenerate metric $(-,-)_{\text{Diff}(S^1)\times_B\Bbb R}$ descends to a metric $(-,-)$ on $(\text{Diff}(S^1)\times_B\Bbb R)/S^1$. By Theorem 1.5, the corresponding Euler equation is
\begin{align*}
    \frac d{dt} (udx\otimes dx,a)&=\text{ad}^*_{A^{-1}(udx\otimes dx,a)}(udx\otimes dx,a)=\text{ad}^*_{(-((\frac {d^2}{dx^2})^{-1}u)\frac d{dx},a)}(udx\otimes dx,a) \\
    &=-2(-((\frac {d^2}{dx^2})^{-1}u)'u-(-((\frac {d^2}{dx^2})^{-1}u)u'-a(-((\frac {d^2}{dx^2})^{-1}u)'''.
\end{align*}
Assuming that $u=v''$, we have
\begin{align*}
    v_{xxt}=2v_{xx}v_x+vv_{xxx}+av_{xxx}.
\end{align*}
Note that if we let $w=v-a$, we could get the Euler equation corresonding to $\text{Diff}(S^1)/S^1$, which has the form
\begin{align*}
   w_{xxt}=2w_{xx}w_x+ww_{xxx}.
\end{align*}

\begin{proposition}
The Euler equation corresponding to the generalized Virasoro group and $H^1$-metric is
\begin{align*}
   \frac d{dt} (u,a)&=\text{ad}^*_{((1-\partial_x^2)^{-1}u\frac d{dx},a)}(udx\otimes dx,a) \\
   &=-2(((1-\partial_x^2)^{-1}u)'u-((1-\partial_x^2)^{-1}u)u'-\alpha a((1-\partial_x^2)^{-1}u)'''+\beta a((1-\partial_x^2)^{-1}u)',a)
\end{align*}
Letting $u=v-v''$, we have
\begin{align*}
   v_t-v_{xxt}=-3v_xv+2v_xv_{xx}+vv_{xxx}-\alpha av_{xxx}+\beta av_x.
\end{align*}
Letting $v=w+\frac13a\beta$, we obtain
\begin{align*}
    w_t-w_{xxt}=-3w_xw+2w_xw_{xx}+ww_{xxx}-(\alpha a+\frac13\beta a)w_{xxx}.
\end{align*}
\end{proposition}

\end{document}